\documentclass[runningheads,a4paper]{llncs}

\usepackage{url}   
\newcommand{\keywords}[1]{\par\addvspace\baselineskip
\noindent\keywordname\enspace\ignorespaces#1}

\usepackage[english]{babel}
\usepackage{times,amsmath,amssymb,graphicx,subfigure, xcolor,environ,tabularx,lipsum,complexity,multirow}
\usepackage{graphics}
\usepackage{setspace}

  \usepackage{tikz}
  \usetikzlibrary{positioning,fit,calc,shapes,backgrounds}
  \usetikzlibrary{matrix}

\makeatletter
\newcommand{\problemtitle}[1]{\gdef\@problemtitle{#1}}
\newcommand{\probleminput}[1]{\gdef\@probleminput{#1}}
\newcommand{\problemquestion}[1]{\gdef\@problemquestion{#1}}
\NewEnviron{myproblem}{
  \problemtitle{}\probleminput{}\problemquestion{}
  \BODY
  \par\addvspace{.5\baselineskip}
  \noindent
  \begin{tabularx}{\textwidth}{@{\hspace{\parindent}} l X c}
    \multicolumn{2}{@{\hspace{\parindent}}l}{\@problemtitle} \\
    \textbf{Input:} & \@probleminput \\
    \textbf{Question:} & \@problemquestion
  \end{tabularx}
  \par\addvspace{.5\baselineskip}
}

\newcommand{\CGCE}{\texttt{CGC-Equality}}
\newcommand{\CGC}{\texttt{CGC-Decision}}

\newcommand*\samethanks[1][\value{footnote}]{\footnotemark[#1]}

\begin{document}

\mainmatter  


\title{Connected greedy coloring $H$-free graphs  }
\titlerunning{Connected greedy coloring H-free graphs}
\authorrunning{Connected greedy coloring H-free graphs}

\author{Esdras Mota\inst{1}\and Ana Silva\inst{2}\thanks{Partially supported by CNPq/Brazil, Project Universal 401519/2016-3, and FUNCAP/CNPq/Brazil, Project PRONEM PNE-0112-00061.01.00/16.}
\and Leonardo Rocha\inst{3}\samethanks}

\institute{Instituto Federal de Educa\c{c}\~{a}o, Ci\^{e}ncia e Tecnologia do Cear\'{a} (UFC) - Quixad\'{a}, CE, Brazil\\
\url{esdras.mota@ifce.edu.br}
\and 
Departamento de Matem\'{a}tica, Universidade Federal do Cear\'{a} (UFC) - Fortaleza, CE, Brazil\\
\url{anasilva@mat.ufc.br}
\and
Universidade Estadual do Cear\'{a} (UECE) - Fortaleza, CE, Brazil\\
\url{leonardo.sampaio@uece.br}
}

\maketitle

\begin{abstract}

A \emph{connected ordering} $(v_1, v_2, \ldots, v_n)$ of $V(G)$ is an ordering of the vertices such that $v_i$ has at least one neighbour in \{$v_1, \ldots, v_{i - 1}$\} for every $i \in \{2, \ldots, n\}$. A \emph{connected greedy coloring} (CGC for short) is a coloring obtained by applying the greedy algorithm to a connected ordering. This has been first introduced in 1989 by Hertz and de Werra, but still very little is known about this problem. An interesting aspect is that, contrary to the traditional greedy coloring, it is not always true that a graph has a connected ordering that produces an optimal coloring; this motivates the definition of the \emph{connected chromatic number of $G$}, which is the smallest value $\chi_c(G)$ such that there exists a CGC of $G$ with $\chi_c(G)$ colors. An even more interesting fact is that $\chi_c(G) \le \chi(G)+1$ for every graph $G$ (Benevides et. al. 2014). 

In this paper, in the light of the dichotomy for the coloring problem restricted to $H$-free graphs given by Kr\'al et.al. in 2001, we are interested in investigating the problems of, given an $H$-free graph $G$: (1). deciding whether $\chi_c(G)=\chi(G)$; and (2). given also a positive integer $k$, deciding whether $\chi_c(G)\le k$.  We have proved that Problem (2) has the same dichotomy as the coloring problem (i.e., it is polynomial when $H$ is an induced subgraph of $P_4$ or of $P_3+K_1$, and it is $\NP$-complete otherwise). As for Problem (1), we have proved that $\chi_c(G) = \chi(G)$ always hold when $G$ is an induced subgraph of $P_5$ or of $P_4+K_1$, and that it is $\NP$-complete to decide whether $\chi_c(G)=\chi(G)$ when $H$ is not a linear forest or contains an induced $P_9$. We mention that some of the results actually involve fixed $k$ and fixed $\chi(G)$.

\end{abstract}


\section{Introduction}

A \emph{$k$-coloring} (or simply a coloring) of a graph $G = (V, E)$ is a surjective function $\psi : V \rightarrow \{1, 2, \dots, k\}$. 
We refer to the values assigned to vertices as their \emph{colors}.
A \emph{proper} coloring is such that $\psi(u)\neq \psi(v)$ for any edge $uv\in E$. 
A proper $k$-coloring can also be given as a partition of the vertex set of $G$ into $k$ disjoint \emph{stable sets} $S_i =\{v \mid \psi(v) = i\}$,  $1\leq i\leq k$. 
In this case, $S_i$ is the \emph{color class}, formed by the vertices $v$ such that $\psi(v) = i$.
The graph is \emph{$k$-colorable} if it admits a proper $k$-coloring.
In what follows, by a coloring we always mean a proper coloring.

Graph colorings are a natural model for problems in which a set of objects is to be partitioned according to some prescribed rules.
Usually, the rules are related to conflicts between the objects to be partitioned.
This model is remarkably useful for \emph{scheduling problems}~\cite{Werra.85}, such as \emph{frequency assignment}~\cite{Gamst.86}, \emph{register allocation}~\cite{Chow.Hennessy.84,Chow.Hennessy.90}, and the \emph{finite element method}~\cite{Saad.96}.

While it is easy to find a coloring when no bound is imposed on the number of color classes, for most of these applications the challenge consists in finding one that minimizes the number of colors.
The \emph{chromatic number} of a graph $G$ is the minimum number of colors in a coloring of $G$; it is denoted by $\chi(G)$ and we say that $G$ is \emph{$k$-chromatic} if $\chi(G) = k$.
An \emph{optimal coloring} is any coloring with $\chi(G)$ colors.  
To decide if a given graph is $k$-colorable is an $\NP$-complete problem, even if $k$ is not part of the input~\cite{Hol81}.
The chromatic number is even hard to approximate: for all $\epsilon > 0$, there is no algorithm that approximates the chromatic number within a factor of $n^{1 - \epsilon}$ unless $\P = \NP$~\cite{Has96,Zuc07}.

Because of the hardness results and the existing practical applications, the study of coloring heuristics is motivated.
In most coloring heuristics, such as the well studied DSATUR~\cite{Bre79}, the \emph{greedy algorithm} is present.
The greedy algorithm works on an input graph $G = (V, E)$ and an ordering $(v_1, \ldots, v_n)$ of its vertices.
For each $i$ from $1$ up to $n$, we color $v_i$ with the smallest color $k \in \{1, \ldots, n\}$ such that no vertex in $N(v_i)$ is already colored $k$.
We call a coloring obtained by the greedy algorithm a \emph{greedy coloring}.

A nice property is that there always exists an ordering that produces an optimal coloring of $G$.
To see this, let $f$ be a $k$-coloring of graph $G$. The order obtained by ordering every vertex of color~1 first, then the vertices of color~2, and so on, clearly produces a coloring that uses at most $k$ colors. Hence, minimizing the number of colors used by the greedy algorithm equals to finding the chromatic number of $G$. 
This is why the parameter studied is related to the worst-case scenario, i.e., the maximum value $k$ for which the greedy algorithm produces a coloring with $k$ colors. This is called the \emph{Grundy number} and is denoted by $\Gamma(G)$. It is a very well studied parameter, but we refrain from citing any work on it here because we will focus on a special type of greedy colorings.

\textbf{Connected greedy colorings.} In this paper we consider a variant of the greedy algorithm called the \emph{connected greedy algorithm}.
A \emph{connected ordering} $(v_1, v_2, \ldots, v_n)$ of $V(G)$ is an ordering of the vertices with the property that $v_i$ has at least one neighbour in \{$v_1, \ldots, v_{i - 1}$\} for every $i \in \{2, \ldots, n\}$.
The \emph{connected greedy algorithm} works similar as the greedy algorithm, but only takes as input connected orderings.
A \emph{connected greedy coloring} is one obtained by the connected greedy algorithm. Up to our knowledge, this has been first introduced in~\cite{hertz1989connected}, where the authors call the connected greedy algorithm {\it SCORE} (Sequential coloring based on a Connected ORdEr).

The concept of \emph{Connected Grundy Number} is naturally defined as being the maximum $\Gamma_c(G)$ for which $G$ has a connected greedy coloring.  But observe that, contrary to the traditional greedy algorithm, it is not obvious that there always exists a connected ordering that produces an optimal coloring. Indeed, in~\cite{babel1994hard}, the authors present a graph on~18 vertices for which this does not hold. They believed that it was the smallest such graph, but up to our knowledge a formal proof has not been given yet. Also, we mention that there is an infinite number of such graphs, as shown in~\cite{BCD+14}. 
Therefore, it makes sense to define the minimization parameter related to connected greedy colorings. The \emph{connected chromatic number} of a graph $G$ is defined as the minimum integer $k$ for which $G$ admits a connected greedy $k$-coloring; it is denoted by $\chi_c(G)$.  
Interestingly enough, this cannot be larger than the chromatic number plus one~\cite{BCD+14}: 
\begin{equation}\chi(G) \leq \chi_c(G) \leq \chi(G) + 1\mbox{, for every graph $G$}.\label{Equation1}\end{equation}

Given graph parameters $a$ and $b$, a graph $G$ is called \emph{$(a,b)$-perfect} if $a(H)=b(H)$  for every induced subgraph $H$ of $G$. In~\cite{ChSe79}, the authors prove that the $(\chi,\Gamma)$-perfect graphs are exactly the cographs. Following this result, concerning connected greedy colorings, in~\cite{hertz1989connected} the authors characterize a subclass of $(\chi,\chi_c)$-perfect graphs (they make other constraints on the order), while in~\cite{KT.18} the authors characterize the claw-free $(\chi,\Gamma_c)$-perfect graphs. We mention that the complexity of recognizing $(\chi,\chi_c)$-perfect, $(\chi_c,\Gamma_c)$-perfect, and $(\chi,\Gamma_c)$-perfect are all open.

Concerning complexity results, in~\cite{BCD+14} it is proved that deciding if $\chi(G)=\chi_c(G)$ is $\NP$-hard, while  in~\cite{BFKS.15} the authors prove that it is $\NP$-complete to decide whether $G$ has a connected greedy coloring with $k$ colors, for every fixed $k\ge 7$. The latter result has been generalized in this article, as will be seeing forward.

Here, we are interested about  whether there exists a hard dichotomy for the $H$-free graphs, as the one below for the traditional coloring problem, proved by  Kr\'al et. al. in~2001 (a graph is \emph{$H$-free} if it does not contain a copy of $H$ as induced subgrah). In what follows, $G+H$ denotes the graph obtained from the union of $G$ and $H$. Also, given a graph $H$, we denote by ${\cal F}(H)$ the class of $H$-free graphs.

\begin{theorem}[\cite{kral2001complexity}]
\label{thmkral}
Let $H$ be fixed. Given $G\in{\cal F}(H)$ and a positive integer $k$, deciding whether $\chi(G)\le k$ can be done in polynomial time if $H$ is an induced subgraph of $P_4$ or $P_3+ K_1$, and is an \NP-complete problem otherwise.
\end{theorem}
\vspace{0.25cm}

Even before this result was presented, it was already known that deciding whether $\chi(G)\le k$ is $\NP$-complete for line graphs for every fixed $k\ge 3$~\cite{Hol81}. Because line graphs are claw-free, it follows that deciding $\chi(G)\le k$ is $\NP$-complete for $H$-free graphs, when $H$ contains a claw. Also, in~2007 Kaminski and Lozin~\cite{KL.07} proved that for every fixed $k\ge 3$ and $g\ge 3$, given a graph $G$ with girth at least~$g$, deciding whether $\chi(G)\le k$ is $\NP$-complete. This gives us that deciding $\chi(G)\le k$ for $G\in {\cal F}(C_\ell)$ is $\NP$-complete for every fixed $k\ge 3$ and every fixed $\ell\ge 3$. Therefore, if the problem is polynomial in ${\cal F}(H)$, then $H$ must be a linear forest (forest of paths). 
This is why much work has been done on the problem of deciding, for fixed values of $k$, whether $\chi(G)\le k$ for graphs with no certain induced paths. It has been proved that it is polynomial-time solvable for $P_5$-free graphs and every positive integer $k$~\cite{HKLSS.10}, and when $k=3$ for $P_6$-free graphs~\cite{RS.04} and, more recently, $P_7$-free graphs~\cite{B.etal.17}. Also, it is $\NP$-complete for $P_6$-free graphs and $k\ge 5$, and for $P_7$-free graphs and $k\ge 4$~\cite{huang2016improved}. Therefore, the only open cases are: 4-coloring $P_6$-free graphs, and 3-coloring $P_\ell$-free graphs, for $\ell\ge 8$. These results are summarized in Table~\ref{table:Pell}.

\begin{table}
\begin{center}
\begin{tabular}{l|p{1cm}p{1cm}p{1cm}p{1cm}}
\hline
\multirow{2}{*}{$P_\ell$-free} & \multicolumn{4}{c}{$k$}\\

 & 3 & 4 & 5 & $\ge 6$ \\
\hline
\hline
$\le 5$ & \P & \P & \P & \P \\
\hline
6          & \P & ? & \NP & \NP \\
\hline
7          & \P & \NP & \NP & \NP\\
\hline
8          & ? & \NP & \NP & \NP\\
\hline
\end{tabular}
\caption{The complexity of deciding $\chi(G)\le k$ for $G\in {\cal F}(P_\ell)$ when $k$ and $\ell$ are fixed.}\label{table:Pell}
\end{center}
\end{table}

Now, coming back to our problem, observe that asking whether $\chi_c(G)=\chi(G)$ is not the same as asking whether $\chi_c(G)\le k$ for a given $k$. Indeed, the former question is not in $\NP$, while the latter is. Below, we formally define these problems.

\begin{myproblem}
  \problemtitle{\texttt{CGC-Equality}}
  \probleminput{A graph $G = (V,E)$.}
  \problemquestion{$\chi_c(G)=\chi(G)$?}
\end{myproblem}

\begin{myproblem}
  \problemtitle{\texttt{CGC-Decision}}
  \probleminput{A graph $G = (V,E)$ and an integer $k$.}
  \problemquestion{$\chi_c(G)\le k$?}
\end{myproblem}

Concercing Problem~{\CGC}, part of our results follow directly from previous ones. This is because of the following easy proposition and the fact that some of the  classes ${\cal F}(H)$ are closed under the addition of an universal vertex. To see that the proposition holds, just observe that $\chi_c(G')=\chi(G)+1$, where $G'$ is obtained from $G$ by adding a universal vertex. 

\begin{proposition}\label{prop:universalvertex}
Let ${\cal G}$ be a graph class and suppose that deciding $\chi(G)\le k$ is \NP-complete if ${G\in\cal G}$ for fixed $k$. If ${\cal G}$ is closed under the addition of a universal vertex, then deciding $\chi_c(G)\le k+1$ is $\NP$-complete on ${\cal G}.$
\end{proposition}

Observe however that, when $H$ is the cycle on~3 vertices or is a claw, then the class ${\cal F}(H)$ is not closed under the addition of universal vertices. Without getting much ahead of ourselves, we mention that we actually investigate the complexity of deciding $\chi_c(G)=k$ for fixed $k$, when $G$ is a $k$-chromatic $H$-free graph. This gives us hardness results for both problems. However, because we were not able to obtain such hardness results for every possible configuration of $H$, we do not have a dichotomy for the Problem~{\CGCE}, while Problem~{\CGC} have the same dichotomy as in Theorem~\ref{thmkral}.

\begin{theorem}\label{dicotomiaDecision}
Let $H$ be a fixed graph, $G\in{\cal H}$ and $k$ be a positive integer. If $H$ is an induced subgraph of $P_4$ or $P_3+ K_1$, then Problem~{\CGC} can be solved in polynomial time. Otherwise, the problem is $\NP$-complete. Furthermore, if $k$ is considered to be fixed and is at least~7, then it remains $\NP$-complete when $H$ is not a linear forest or contains a $P_6$ as induced subgraph.
\end{theorem}

\begin{theorem}\label{dicotomiaEquality}
Let $H$ be a fixed graph and $G\in{\cal F}(H)$. If $H$ is not a linear forest or $H$ contains a $P_9$ as induced subgraph, then {\CGCE} is $\NP$-hard.  Also,  if $H$ is an induced subgraph of $P_5$ or of $P_4+K_1$, then $\chi_c(G)=\chi(G)$. 
\end{theorem}

Observe that the polynomial case of {\CGCE} in the theorem above consists of a very simple algorithm: it always says ``yes" if $G$ is $H$-free, for $H$ an induced subgraph of $P_5$ or $P_4+K_1$. We ask whether this is always the case when $H$ is a path:

\begin{question}
Does there exist $\ell$ such that $\chi_c(G) = \chi(G)$ for every $G\in \bigcup_{j\le \ell}{\cal F}(P_j)$, while deciding whether $\chi_c(G) = \chi(G)$ is $\NP$-hard for every $G\in \bigcup_{j > \ell}{\cal F}(P_j)$?
\end{question}

Also, as already mentioned, we actually prove that deciding $\chi_c(G) = \chi(G)$ is $\NP$-hard even if $G$ is a $k$-chromatic graph, for some fixed values of $k$. The only proof where we did not succeed in fixing $k$ was for the $P_9$-free graphs. Therefore, we ask:
\begin{question}
For fixed $k$, given a $P_9$-free $k$-chromatic graph $G$, can one decide in polynomial time whether $\chi_c(G)=k$?
\end{question}




Now, concerning only $P_\ell$-free graphs, from Table~\ref{table:Pell} and Proposition~\ref{prop:universalvertex} we get the situation depicted in Table~\ref{table:Pellconnected}. Position $k,\ell$ of this table tells us the complexity of deciding $\chi_c(G)\le k$ for $G\in {\cal F}(P_\ell)$. The $\NP$-completeness results propagate along the rows because of the proposition above, and along the columns because ${\cal F}(P_\ell)\subseteq {\cal F}(P_{\ell+1})$. The row related to ${\cal F}(P_5)$ is entirely polynomial because of the result in~\cite{HKLSS.10} previously mentioned and by Theorem~\ref{dicotomiaDecision}. The question marks are open problems and they propagate along the rows in a column.

\begin{table}
\begin{center}
\begin{tabular}{l|p{1cm}p{1cm}p{1cm}p{1cm}}
\hline
\multirow{2}{*}{$P_\ell$-free} & \multicolumn{4}{c}{$k$}\\

 & 3 & 4 & 5 & $\ge 6$ \\
\hline
\hline
$\le 5$ & \P & \P & \P & \P \\
\hline
6          & ? & ? & ? & \NP \\
\hline
7          & ? & ?& \NP & \NP\\
\hline
8          & ? & ? & \NP & \NP\\
\hline
\end{tabular}
\caption{The complexity of deciding $\chi_c(G)\le k$ for $G\in {\cal F}(P_\ell)$ when $k$ and $\ell$ are fixed.}\label{table:Pellconnected}
\end{center}
\end{table}

%

We mention that to prove Theorems~\ref{dicotomiaDecision} and~\ref{dicotomiaEquality}, we actually investigate the edge version of the problems. Since every line graph is claw-free, we get that the problem is \NP-hard for claw-free graphs as well, which means that if $H$ is a non-linear forest, then the problem is $\NP$-hard on $H$-free graphs. Let $\chi'_c(G)$ denote the \emph{connected chromatic index of $G$} (which equals the connected chromatic number of $L(G)$, the line graph of $G$). Observe that, by Vizing's Theorem and Equation~\ref{Equation1}, we get $\chi'_c(G)\in \{\Delta(G),\Delta(G)+1,\Delta(G)+2\}$, for every graph $G$. However, we were not able to find a graph $G$ with $\chi'_c(G)=\Delta(G)+2$; note that such a graph would necessarily be a Class~2 graph (a graph is \emph{Class 2} if $\chi'(G)=\Delta(G)+1$). Hence, we pose the following question.

\begin{question}
Does there exist a Class 2 graph $G$ such that $\chi'_c(G)>\chi(G)$?
\end{question}

Another aspect that relates to our investigation is the notion of hard-to-color graphs. In~\cite{babel1994hard}, a connected graph $G$ is called \emph{globally hard-to-color} if, for every $v\in V(G)$ and every $\alpha\in \{1,\ldots,\chi(G)\}$, if $G$ is greedily colored using a connected  order starting at $v$ with color $\alpha$, then the produced coloring uses more than $\chi(G)$ colors. To prove the hardness results for $P_9$-free graphs, we construct a globally hard-to-color $P_9$-free graph. This leads us to the following question (by our results, we know that the answer is in $\{6,7,8,9\}$):

\begin{question}
What is the smallest $\ell$ such that there exists a globally hard-to-color $P_\ell$-free graph?
\end{question}

Finally, although not directly related to the studied problem, we would like to expose another interesting and non-trivial aspect of greedy colorings in order to pose one last question. In~\cite{ChSe79}, the authors prove that for every integer $k\in \{\chi(G),\ldots,\Gamma(G)\}$, there exists a greedy coloring of $G$ with $k$ colors. We ask whether the same holds for connected greedy colorings.

\begin{question}
Let $G$ be any graph. Does there always exist a connected greedy coloring with $k$ colors, for every $k\in \{\chi_c(G),\ldots,\Gamma_c(G)\}$?
\end{question}

Many other questions can be posed on these colorings since, up to our knowledge, very little is known about both parameters, with only the aforementioned articles having been published on the subject. 

In Section~\ref{sec:outline}, we briefly discuss the proof of Theorems~\ref{dicotomiaDecision} and~\ref{dicotomiaEquality} as a whole. In Sections~\ref{sec:cycles} through~\ref{sec:paths}, we prove the $\NP$-hardness results, and in Section~\ref{sec:P5}, we prove that $\chi_c(G)=\chi(G)$ when $G$ is either a $P_5$-free graph or a $(P_4+K_1)$-free graph. 

%
%
%
%
%


\section{Outline of the proof}\label{sec:outline}

First, observe that, for $k\ge 4$, the graph classes ${\cal F}(C_k)$ and ${\cal F}(P_k)$ are closed under the addition of a universal vertex. Therefore, by Proposition~\ref{prop:universalvertex} and Theorem~\ref{thmkral}, we get that {\CGC} is $\NP$-complete on ${\cal F}(C_k)$ and on ${\cal F}(P_\ell)$, for $k\ge 4$ and $\ell\ge 5$. In fact, we get stronger results since these problems are $\NP$-complete even for fixed $k$, as we mentioned before. 
However, these results do not imply that {\CGCE} is also $\NP$-hard on these graph classes. Here, we actually prove the following results, which imply $\NP$-hardness of both {\CGCE} and {\CGC}.

\begin{lemma}\label{lem:CellFree}
For every $k\ge 5$ and $\ell\in\{3,5\}$, deciding if $\chi_c(G)= k$ is $\NP$-complete even when restricted to $k$-chromatic graphs in ${\cal F}(C_\ell)$. 
\end{lemma}

We mention that the hardness for $\ell\notin\{3,5\}$ follows from a result in~\cite{BCD+14}.

\begin{lemma}\label{lem:line}
For every $k\ge 7$, deciding if $\chi'_c(G)=k$ is $\NP$-complete for $C_3$-free $k$-edge-chromatic graphs.
\end{lemma}

These lemmas finish the proof of Theorem~\ref{dicotomiaDecision}. 
As for Theorem~\ref{dicotomiaEquality}, it remains to study the case where $H$ is a linear forest, which is done in the following lemmas.

\begin{lemma}\label{lem:paths}
Deciding if $\chi_c(G)=\chi(G)$ is $\NP$-hard for $P_\ell$-free graphs, for every $\ell\ge 9$.
\end{lemma}

\begin{lemma}\label{lem:poly}
If $H$ is an induced subgraph of $P_5$ or of $P_4+K_1$, then $\chi_c(G)=\chi(G)$, for every $H$-free graph $G$.
\end{lemma}

Before, we begin, we need one last definition. Given a vertex $v\in V(G)$, an integer $\alpha$ and a coloring $f$ of $G$, we say that $f$ is a \emph{$(v,\alpha)$-connected greedy coloring of $G$} (or $(v,\alpha)$-CGC for short) if there exists a connected order of $V(G)$, $\pi$, such that $\pi$ starts with $v$ and $f$ is obtained by coloring $v$ with $\alpha$ and applying the greedy algorithm on $\pi$. The similar is used for edge-coloring. Also, from now on we call a connected greedy coloring and a connected greedy edge-coloring by CGC and ECGC, for short	.


\section{$C_\ell$-free graphs}\label{sec:cycles}

Here, we prove Lemma~\ref{lem:CellFree}. It is known that deciding whether $\chi(G)\le 3$ is $\NP$-complete for planar graphs~\cite{GJS76}. Observe that this and the Four Color Theorem imply that the following problem is also $\NP$-complete for every fixed $k\ge 4$. It suffices to add $k-3$ universal vertices to a planar graph $G$ to obtain $G'$ such that $\chi(G)\le 3$ if and only if $\chi(G') \le k$. 

\ifdefined\comment\comment{
Acho que  reducao no artigo do LATIN serve pra mostrar que o problema eh $\NP$-completo para grafos livres de $C_\ell$, para todo $\ell\neq 3,5$. Por isso, aqui me restringi somente a esses casos particulares.
}\else\fi

\begin{myproblem}
  \problemtitle{\texttt{$k$-COL $(k+1)$-COLORABLE}}
  \probleminput{A graph $G = (V,E)$ with a universal vertex such that $\chi(G)\le k+1$.}
  \problemquestion{$\chi(G)\le k$?}
\end{myproblem}


Given an instance $G$ of the problem above, we construct a $(k+1)$-chromatic $C_\ell$-free graph $G^{**}$ such that $\chi(G)\le k$ if and only if $\chi_c(G^{**})= k+1$. Lemma~\ref{lem:CellFree} follows because Problem \texttt{$k$-COL $(k+1)$-COLORABLE} is $\NP$-complete for every $k\ge 4$. We first construct a gadget that will admit cycles of undesired length, and afterwards we replace some of the edges in order to get rid of such cycles. Start with three disjoint cliques $U,V,M$ each of size $k$, and let $w$ be a vertex  of $M$. Obtain $G_k$ by adding vertices $u,u',v,v'$ and making $u,u'$ complete to $U$, $v,v'$ complete to $V$, $u,v$ complete to $M\setminus\{w\}$, and $u',v'$ complete to $\{w\}$. See Figure~\ref{fig:Cellgadget} for the construction of $G_3$. 

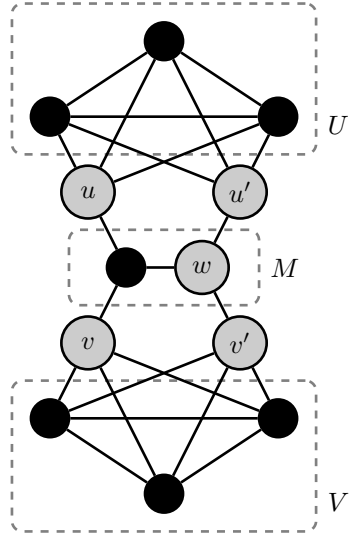
\begin{figure}[thb]
\begin{center}
  \begin{tikzpicture}[scale=1]
  \pgfsetlinewidth{1pt}

  \tikzset{vertex/.style={circle, minimum size=0.7cm, fill=black!20, draw, inner sep=1pt}}
  \tikzset{simple/.style={circle, minimum size=0.5cm, fill=black, draw, inner sep=1pt}}

   \node [simple] (u1) at (-1.5,2){};
   \node [simple] (u2) at (1.5,2){};
   \node [simple] (u3) at (0,3){};
   
   \node [label=350:$U$, draw=black!50, rounded corners, dashed, minimum width=4cm, minimum height=2cm] at (0,2.5) {};

  \node [vertex] (u) at (-1,1){$u$};
  \node [vertex] (up) at (1,1){$u'$};
  
  \node [simple] (a) at (-0.5,0){};
   \node [vertex] (b) at (0.5,0){$w$};
   \node [label=0:$M$,draw=black!50, rounded corners, dashed, minimum width=2.5cm, minimum height=1cm] at (0,0) {};
   
   \node [vertex] (v) at (-1,-1){$v$};
  \node [vertex] (vp) at (1,-1){$v'$};
   
\node [simple] (v1) at (-1.5,-2){};
   \node [simple] (v2) at (1.5,-2){};
   \node [simple] (v3) at (0,-3){};
   
   \node [label=350:$V$,draw=black!50, rounded corners, dashed, minimum width=4cm, minimum height=2cm] at (0,-2.5) {};

	\draw (u) edge (a) 	(up) edge (b)	(v) edge (a)	(vp) edge (b)	(a) edge (b)
	(u1) edge (u2)   (u1) edge (u3)   (u2) edge (u3)
	(v1) edge (v2)   (v1) edge (v3)   (v2) edge (v3)
	(u) edge (u1)   (u) edge (u2)   (u) edge (u3)
	(up) edge (u1)   (up) edge (u2)   (up) edge (u3)
	(v) edge (v1)   (v) edge (v2)   (v) edge (v3)
	(vp) edge (v1)   (vp) edge (v2)   (vp) edge (v3);
	
   \draw[black] (0,-1.2) node [below] {};

  \end{tikzpicture}
\caption{Gadget $G_3$.}
\label{fig:Cellgadget}
\end{center}
\end{figure}

Now, given a $(k+1)$-colorable graph $G$ that contains a dominant vertex, denote by $G^*$ the graph obtained from $G$ by appending a copy of $G_k$ on each vertex of $G$, identifying on vertex $u$ of $G_k$. We first prove that $\chi(G)\le k$ if and only if $\chi_c(G^*) = \chi(G^*) = k+1$, and then we show how to get rid of the undesired cycles. This is actually the proof presented in~\cite{BCD+14}. Note that, because deciding $\chi(G)\le k$ for a $C_\ell$-free graph is $\NP$-complete for every fixed $k\ge 3$ and $\ell\ge 3$~\cite{KL.07} and by previously mentioned aspects, this gives us that \texttt{CGC-Decision} is $\NP$-complete even when restricted to $C_\ell$-free $k$-chromatic graphs, for every ($\ell=4$ or $\ell\ge 6$) and $k\ge 4$. For the first part of the proof, it is essential to understand the following properties of graph $G_k$.

\begin{lemma}
Let $G_k$ be the graph obtained as above, where $k$ is a positive integer, $k\ge 3$.
\begin{enumerate}
  \item $\chi(G_k)=k+1$;
  \item In every $(k+1)$-coloring of $G_k$, vertices $u,v,u',v'$ receive the same color;
  \item For every $x\in V(G_k)$ and every $\alpha\in\{1,\ldots,k\}$, we get that vertices $u,v,u',v'$ receive color at most $k$ in every $(x,\alpha)$-CGC of $G_k$ with $k+1$ colors.
\end{enumerate}
\end{lemma}
\begin{proof}
The first and second properties can be easily verified: it suffices to see that the colors used in $\{u,u',v,v'\}$ cannot be used in any of the cliques $U,V,M$. For the last property, consider any connected order starting with $x$ and suppose, without loss of generality, that $x\notin V\cup\{v,v'\}$. Then, no matter which one between $v$ and $v'$ gets colored first, we get that it will have at most $k-1$ colored neighbors and, therefore, cannot have color bigger than $k$. The property then follows from Property 2.\end{proof}

Now, observe that if $\chi_c(G^*)=k+1$, then Property 3 trivially implies that $\chi(G)\le k$. On the other hand, if $\chi(G)=k$, because $G$ has an universal vertex, we get that $\chi_c(G)=k$. A CGC of $G^*$ with $k+1$ colors can be easily constructed from a CGC of $G$ with $k$ colors. This proves that $\chi(G)\le k$ if and only if $\chi_c(G^*)=k+1$, as we wanted to show.

\ifdefined\comment\comment{
Eu acho que bastava essa segunda parte na verdade, pois ela reduz o problema de decidir $\chi_c(G)=k$ para $G$ $k$-cromático no problema de decidir $\chi_c(G)=k$ para $G$ $k$-cromático sem ciclos de tamanho 3 ou sem ciclos de tamanho 5. FALSO Eh essencial que as arestas incidentes nos vertices especias de $G_k$ nao sejam substituidas, pois senao a propriedade 3 anterior deixa de valer.
}\else\fi

Now, we construct gadgets that will replace the edges of $G^*$ so as to ensure that the obtained graph has no cycles of length $\ell$, for $\ell=3$ and $\ell=5$. In fact, we need a different gadget for each case because, when avoiding cycles of length 3, we end up creating cycles of length 5, and vice-versa. First, we show the gadget necessary for the case $\ell=5$.

Let $G^5_k(p,q)$ be obtained as follows. Start with a $P_4$, $(p,x,y,z)$, a clique $P$ of size $k-1$, and a clique $Q$ of size $k$. Let every vertex in $\{p,x,y,z\}$ be complete to $P$, and let $z$ be complete to $Q$. Finally, add vertex $q$ and make it complete to $Q$. Figure~\ref{fig:G53pq} depicts $G^5_3(p,q)$.

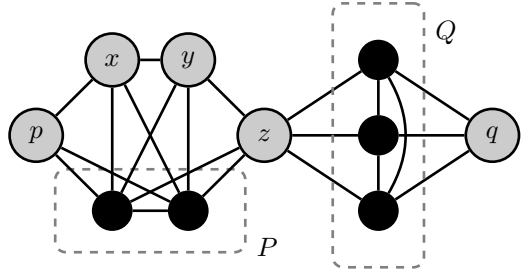
\begin{figure}[thb]
\begin{center}
  \begin{tikzpicture}
  \pgfsetlinewidth{1pt}

  \tikzset{vertex/.style={circle, minimum size=0.7cm, fill=black!20, draw, inner sep=1pt}}
  \tikzset{simple/.style={circle, minimum size=0.5cm, fill=black, draw, inner sep=1pt}}
      
    \node [vertex] (p) at (-3,0){$p$};
    \node [vertex] (x) at (-2,1){$x$};
    \node [vertex] (y) at (-1,1){$y$};
    \node [vertex] (z) at (0,0){$z$};
    \node [simple] (p1) at (-1,-1){};
    \node [simple] (p2) at (-2,-1){};
    \node [label=350:$P$,draw=black!50, rounded corners, dashed, minimum width=2.5cm, minimum height=1.1cm] at (-1.5,-1) {};
    
    \node [vertex] (q) at (3,0){$q$};
    \node [simple] (q1) at (1.5,1){};
    \node [simple] (q2) at (1.5,-1){};
    \node [simple] (q3) at (1.5,0){};    
    \draw [-] (q1) to [out=300,in=60] (q2);
    
    \draw (p) edge (x)  (x) edge (y)   (y) edge (z)  (z) edge (q3)
    (z) edge (q1) (z) edge (q2) (z) edge (p1) (z) edge (p2)
    (q) edge (q1) (q) edge (q2) (q) edge (q3)
    (p) edge (p1) (p) edge (p2)
    (x) edge (p1) (x) edge (p2)
    (y) edge (p1) (y) edge (p2)
    (p1) edge (p2) (q1) edge (q3) (q2) edge (q3);

      	\node [label=60:$Q$,draw=black!50, rounded corners, dashed, minimum height=3.5cm, minimum width=1.2cm] at (1.5,0) {};
  \end{tikzpicture}
\caption{Gadget to replace the edges on $G^*$ when $\ell=5$; denoted by $G^5_3(p,q)$.}
\label{fig:G53pq}
\end{center}
\end{figure}

Now, let $G^{**}$ be obtained from $G^*$ by replacing every edge $pq$ with a copy of $G^5_k(p,q)$, except the edges of $G^*$ incident to the copies of $\{u,v,u',v'\}$ in the gadgets of type $G_k$. Observe that $G^5_k(p,q)$ has no induced cycles of length~5; also, because the distance between $p$ and $q$ in $G_k(p,q)$ is~4, we get that $G^{**}$ also does not have cycles of length~5. Now, we prove that $\chi_c(G^*)=k+1$ if and only if $\chi_c(G^{**})=k+1$, thus finishing the proof of the case $\ell=5$. For this, we need the following properties to hold.

\begin{lemma}
Let $G^5_k(p,q)$ be obtained as above, where $k$ is a positive integer, $k\ge 3$.
\begin{enumerate}
\item[(i)] $\chi(G^5_k(p,q))=k+1$ and in every $(k+1)$-coloring of $G^5_k(p,q)$, vertices $p$ and $q$ receive distinct colors;

\item[(ii)] For every $\alpha,\beta\in\{1,\ldots,k+1\}$, $\alpha\neq\beta$, there exists a $(p,\alpha)$-CGC of $G^5_k(p,q)$ in which $q$ is colored with $\beta$, and a $(q,\beta)$-CGC in which $p$ is colored with $\alpha$.
\label{lem:Gkpq}
\end{enumerate}
\end{lemma}
\begin{proof}
To see that $\chi(G^5_k(p,q))=k+1$ just observe that $\omega(G^5_k(p,q)) = k+1$ and that a $(k+1)$-coloring of $G^5_k(p,q)$ can be obtained by giving colors $\{3,\ldots,k+1\}$ to $P$, color 2 to $\{p,y\}$, color 1 to $\{x,z,q\}$, and colors $\{2,\ldots,k+1\}$ to $Q$. Now, let $f$ be a $(k+1)$-coloring of $G^5_k(p,q)$. Because $z$ and $q$ are complete to $Q$ of size $k$, they receive the same color; similarly, $x$ and $z$ also receive the same color, and since $p$ is adjacent to $x$, Property (i) follows. Now, given $\alpha,\beta\in\{1,\ldots,k+1\}$, $\alpha\neq\beta$, we construct the desired $(p,\alpha)$-CGC as follows: color $p$ with $\alpha$ and enough vertices of $P$ so that colors $1$ through $\beta-1$ appear in $P\cup \{p\}$; then, color $x$ with $\beta$. If $\alpha\in\{1,\ldots,\beta-1\}$, color $y$ with $\alpha$, $z$ with $\beta$, and finish coloring $P$. Otherwise, color $z$ with $\beta$, some vertices of $P$ with $\{\beta+1,\ldots,\alpha-1\}$, $y$ with $\alpha$, and finish coloring $P$. It is easy to see that this can be extended to $Q\cup\{q\}$ as desired, and also that we could have started in $q$ and made our way to $p$ (this is because vertices $p$ and $z$ are symmetric with relation to $P\cup\{x,y\}$).
\end{proof}

Observe that Property (ii) makes it possible to construct a CGC of $G^{**}$, given a CGC of $G^*$; hence, if $\chi_c(G^{*})=k+1$, then $\chi_c(G^{**})=k+1$. Also, note that Property (i) and the fact that the edges incidents to the copies of vertices $\{u,u',v,v'\}$ were not replaced ensures that Properties (1)-(3) still hold on the copies of $G_k$. Therefore, if $\chi_c(G^{**})=k+1$, then Property (3) ensures us that $\chi(G) \le k$, and the previous argument gives us that $\chi_c(G^*)=k+1$.

Now, we present the gadget needed for the case $\ell=3$. We prove that Lemma~\ref{lem:Gkpq} also holds for the consutrcted gadget, thus finishing our proof since the same arguments can be applied. For this, we introduce an operation on graphs. Let $H$ be any graph. The \emph{double myscielskian of $H$}  is the graph $H^*$ obtained from $H$ as follows (observe Figure~\ref{fig:G2pq} to see the operation applied to the $P_4$ $(p,x,y,q)$): add two copies $V_1$ and $V_2$ of $V(H)$, and denote the copy of $w$ in $V_i$ by $w_i$; for every edge $yz\in E(G)$, add edges $y_iz$ and $yz_i$, for $i=1$ and $i=2$; finally, add two adjacent vertices $x_1$ and $x_2$ and make $x_i$ complete to $V_i$, for $i=1$ and $i=2$. It is well known that the Mycieslki of a triangle-free $h$-chromatic graph produces a triangle-free $(h+1)$-chromatic graph~\cite{M55}. Also, if $\chi(H)\ge 2$, an $(h+1)$-coloring of $H^*$ can be constructed from an optimal coloring of $H$ by coloring $V_1\cup V_2$ with $h+1$ and $x_i$ with $i$, $i\in \{1,2\}$; hence $\chi(H^*)=\chi(H)+1$. Now, let $G^3_k(p,q)$ be obtained by applying $k-1$ times the double mycielskian operation, starting with the $P_4$ $(p,x,y,q)$. By what was said before, we get that $\chi(G^3_k(p,q)) = k+1$ and that $G_k(p,q)$ is triangle-free. Also, since $p$ and $q$ are not adjacent, the graph $G^{**}$ constructed as before is also triangle-free. It remains to prove Properties (i) and (ii).

\begin{figure}[thb]
\begin{center}
  \begin{tikzpicture}
  \pgfsetlinewidth{1pt}
  \tikzset{vertex/.style={circle, minimum size=0.7cm, fill=black!20, draw, inner sep=1pt}}

    \node [vertex] (p) at (0,0){$p$};
    \node [vertex] (x) at (0,1){$x$};
    \node [vertex] (y) at (0,2){$y$};
    \node [vertex] (q) at (0,3){$q$};

    \node [vertex] (p1) at (-2,0){$p_1$};
    \node [vertex] (x1) at (-2,1){$x_1$};
    \node [vertex] (y1) at (-2,2){$y_1$};
    \node [vertex] (q1) at (-2,3){$q_1$};
    \node [vertex] (t1) at (-4,1.5){$t_1$};
    \node [label=90:$V_1$,draw=black!50, rounded corners, dashed, minimum height=4cm, minimum width=1.2cm] at (-2,1.5) {};

    \node [vertex] (p2) at (2,0){$p_2$};
    \node [vertex] (x2) at (2,1){$x_2$};
    \node [vertex] (y2) at (2,2){$y_2$};
    \node [vertex] (q2) at (2,3){$q_2$};
    \node [vertex] (t2) at (4,1.5){$t_2$};  
    \node [label=90:$V_2$,draw=black!50, rounded corners, dashed, minimum height=4cm, minimum width=1.2cm] at (2,1.5) {};
  
    \draw (p) edge (x) (x) edge (y) (y) edge (q);
    \foreach \i/\label in {1, 2}{
    	\draw (p\i) edge (x) (p) edge (x\i) (x\i) edge (y) (x) edge (y\i) (y\i) edge (q) (y) edge (q\i)
    	(t\i) edge (p\i)  (t\i) edge (x\i)  (t\i) edge (y\i)  (t\i) edge (q\i);
    }
    \draw [-] (t1) to [out=270,in=270] (t2);

  \end{tikzpicture}
\caption{Edge gadget for $C_3$-free graphs. In the figure, only one application of the double Mycielski is made, i.e., it depicts the graph $G^3_2(p,q)$.}
\label{fig:G2pq}
\end{center}
\end{figure}
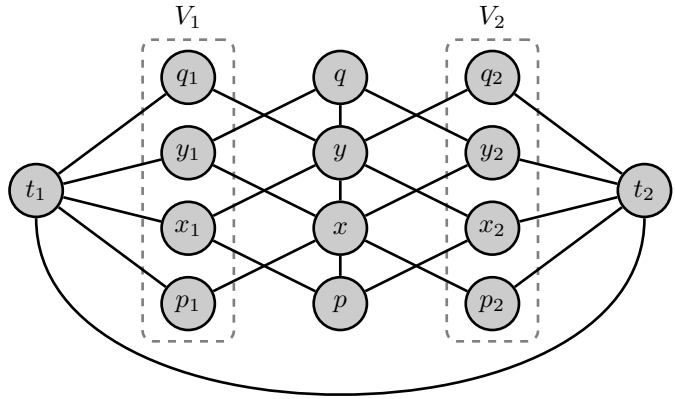

\begin{lemma}
Let $G^3_k(p,q)$ be obtained as above, where $k$ is a positive integer, $k\ge 3$.
\begin{enumerate}
\item[(i)] $\chi(G^3_k(p,q))=k+1$ and in every $(k+1)$-coloring of $G^3_k(p,q)$, vertices $p$ and $q$ receive distinct colors;

\item[(ii)] For every $\alpha,\beta\in\{1,\ldots,k+1\}$, $\alpha\neq\beta$, there exists a $(p,\alpha)$-CGC of $G^3_k(p,q)$ in which $q$ is colored with $\beta$, and a $(q,\beta)$-CGC in which $p$ is colored with $\alpha$.
\label{lem:Gkpq}
\end{enumerate}
\end{lemma}
\begin{proof}
Let $t_1,t_2,V_1,V_2$ be the additional vertices added at the last application of the double Myscielski operation. To see that $\chi(G^3_k(p,q))=k+1$, recall that the Mycielski operation increases the chromatic number by one. Hence, the double Mycielski operation increases by at least one; the fact follows because a $(k+1)$-coloring can be obtained from a $k$-coloring of $G^3_{k-1}(p,q)$ by giving a new color to $V_1\cup V_2$ and using colors 1 and 2 in $t_1,t_2$. 
Now, we prove the second part of Property (i) by induction on $k$. Note that it trivially holds for the initial $P_4$. So, suppose by contradiction that $f$ is a $(k+1)$-coloring of $G^3_k(p,q)$ such that $f(p)= f(q)=1$. Because $f(t_1)\neq f(t_2)$, we can suppose that $f(t_1)=j\neq 1$. Note that, by switching the color of every $w\in V(G^3_{k-1}(p,q))$ such that $f(w)=j$ to $f(w_1)$, we obtain a $k$-coloring of $G^3_{k-1}(p,q)$: the color of $w_1$ certainly does not appear in $N(w)\cap V(G^3_{k-1}(p,q))$, and we only change the colors of vertices contained in a stable set. But since $j\neq 1$, this coloring is a $k$-coloring of $G^3_{k-1}(p,q)$ in which $p$ and $q$ receive the same color, contradicting the induction hypothesis. 

Finally, consider $\alpha,\beta\in\{1,\ldots,k+1\}$, $\alpha\neq\beta$; we prove Property (ii) also by induction on $k$. Observe that, because $p$ and $q$ are symmetric in $G^3_k(p,q)$, we only need to prove the existence of a $(p,\alpha)$-CGC in which $q$ is colored with $\beta$. If $k=1$, then $\{\alpha,\beta\}=\{1,2\}$ and the property trivially holds. Now, suppose it holds for $k-1$ and consider the following cases:

\begin{itemize}
\item $\alpha,\beta<k+1$: let $f$ be a $(p,\alpha)$-CGC of $G^3_{k-1}(p,q)$ in which $f(q)=\beta$. Then, let $f'$ be obtained from $f$ by coloring $w_1$ with $f(w)$, for every $w\in V(G^3_{k-1}(p,q))$, $t_1$ with $k+1$, $t_2$ with 1 and $w_2$ with $\min\{c\mid c\notin f'(N(w_2))\}$, for every $w\in V(G^3_{k-1}(p,q))$ (observe that this is at most $k+1$);

\item $\beta=k+1$: if $\alpha=1$, let $j=2$ and $f$ be a $(p,1)$-CGC in which $q$ is colored with $k$; and if $\alpha>1$,  let $j=1$ and $f$ be a $(p,\alpha-1)$-CGC in which $q$ is colored with $k$. Also, let $w$ be any neighbor of $p$ in $G^3_{k-1}(p,q)$. Give color $\alpha$ to $p$, color $j$ to $w_1$, color $3-j$ to $t_1$, and color $j$ to $V_1$. Let $(p=v_0,v_1,\ldots,v_n)$ be the order that produces $f$, and denote by $f'$ the partial coloring of $G^3_k(p,q)$. Now, for each $h\in\{1,\ldots,n\}$, if $j=2$  and $f(v_h)=1$, then let $f'(v_h)=1$; otherwise, let $f'(v_h)=f(v_h)+1$. Because $f'(v_m)\ge f(v_m)$ for every $m$, we know that if $f(v_h)=1$ and $j=2$, then color 1 does not appear in $N(v_h)$ and color 1 is allowed for $v_h$. Otherwise, because for each $c\in \{2,\ldots,f(v_h)-1\}$ the neighbor of $v_h$ in $f^{-1}(c)$ is now colored with $c+1$ (including $p$ if $\alpha>1$), we just need to prove that $v_h$ also has neighbors of colors 1 and 2 in $f'$ in order to prove that $v_h$ is greedily colored. Let $v_m\in N(v_h)\cap f^{-1}(1)$. If $j=2$, then $f'(v_m)=1$ and $v_h$ has a neighbor in $V_1$ of color~2; otherwise, $f'(v_m)=2$ and $v_h$ has a neighbor of color~1 in $V_1$. To finish coloring $G^3_k(p,q)$, give color $\min\{c\mid c\notin f'(N(h_2))\}$, for every $h\in V(G^3_{k-1}(p,q))$ (observe that this is at most $2$), and color 3 to $x_2$. Because $k\ge3$, we know that $f'(q) = f(q)+1 = k+1$ as desired.

\item $\alpha = k+1$: a similar argument can be applied by getting a $(p,k)$-CGC of $G^3_{k-1}(p,q)$ in which $q$ is colored with $\beta-1$, when $\beta>1$, or in which $q$ is colored with~1, otherwise. 
\end{itemize}
\end{proof}

\section{Line graphs}

\ifdefined\comment\comment{O artigo citado pelo Esdras nao considera $k$ fixo e o artigo original faz apenas para $k=3$, o grafo sendo 3-regular.}\else\fi

Here we prove Lemma~\ref{lem:line} by making a reduction from the problem of deciding whether $G$ is 3-edge-colorable when $G$ is a triangle-free cubic graph, which is known to be $\NP$-complete~\cite{Hol81}. The idea follows the one applied for $C_\ell$-free graphs: for each $k\ge 7$, given an instance $G$ of the problem above, we append on each vertex some copies of a gadget that ensures that the obtained graph $G^*$ is such that $\chi'_c(G^*) = \chi'(G^*) = k$ if and only if $\chi'(G)\le 3$. For this, we first construct what we call \emph{edge gadgets}.

For each $k\ge 7$, let $G'_k(p,q)$ be obtained from a complete bipartite graph $K_{k-1,k-1}$ with parts  $P$ and $Q$ by adding $p,q,p',q'$, making $p'$ complete to $P$, $q'$ complete to $Q$ and adding edges $pp'$ and $qq'$. Graph $G'_3(p,q)$ is depicted in Figure~\ref{fig:Gp3pq}. 

\begin{figure}[thb]
\begin{center}
  \begin{tikzpicture}[scale=0.9]
  \pgfsetlinewidth{1pt}
  \tikzset{vertex/.style={circle, minimum size=0.7cm, fill=black!20, draw, inner sep=1pt}}
  \tikzset{simple/.style={circle, minimum size=0.5cm, fill=black, draw, inner sep=1pt}}

  \node [vertex] (p) at (-4,0){$p$}; 
  \node [vertex] (pp) at (-2.5,0){$p'$};
  \node [vertex] (q) at (4,0){$q$}; 
  \node [vertex] (qp) at (2.5,0){$q'$};
  
  \node [simple] (p1) at (-1,1){}; 
  \node [simple] (p2) at (-1,-1){};
  \node [label=270:$P$,draw=black!50, rounded corners, dashed, minimum height=3cm, minimum width=1.2cm] at (-1,0) {};
  \node [simple] (q1) at (1,1){}; 
  \node [simple] (q2) at (1,-1){};
  \node [label=270:$Q$,draw=black!50, rounded corners, dashed, minimum height=3cm, minimum width=1.2cm] at (1,0) {};
  
  \draw (p)--(pp)--(p1)--(q1)--(qp)--(q)
  (pp)--(p2)--(q2)--(qp)
  (p1)--(q2)   (p2)--(q1);

  \end{tikzpicture}
\caption{Edge-gadget in the reduction for line graphs; $G'_3(p,q)$ is depicted.}
\label{fig:Gp3pq}
\end{center}
\end{figure}
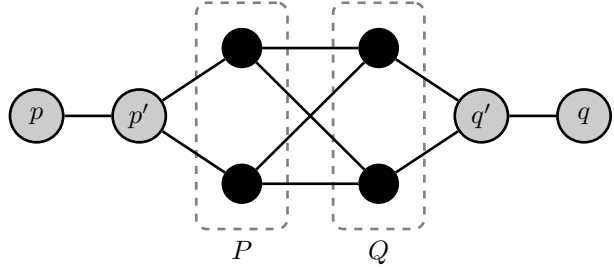

Now, we construct the gadget that will be appended on the vertices of $G$. Observe Figure~\ref{fig:linegadget} to follow the construction. Consider $k$ copies of the edge gadget and denote them by $G'_k(p,q)$ and $G'_k(p_i,q_i)$, $i\in\{1,\ldots,k-1\}$. Let $G'_k$ be obtained by identifying vertices $\{p_1,\ldots,p_{k-1}\}$ into vertex $w$, vertices $\{q,q_1,q_2,q_3\}$ into vertex $u$, vertices $\{p, q_4,\ldots,q_{k-1}\}$ into vertex $v$, and finally adding a new vertex $t$ adjacent to $w$. 

\begin{figure}[thb]
\begin{center}
  \begin{tikzpicture}
  \pgfsetlinewidth{1pt}
  \tikzset{vertex/.style={circle, minimum size=0.7cm, fill=black!20, draw, inner sep=1pt}}
  \tikzset{simple/.style={circle, minimum size=0.2cm, fill=black, draw, inner sep=1pt}}
  \tikzset{ponto/.style={circle}}
  \tikzset{XY/.style={fill=white, draw=black!50, ellipse, dashed, minimum height=0.5cm, minimum width=1cm}}
\tikzset{YX/.style={fill=white, draw=black!50, ellipse, dashed, minimum height=1cm, minimum width=0.5cm}}

  \draw (-5.4,0)--(-5.4,-0.5)  (-5.7,0)--(-5.7,-0.5) (-5.1,0)--(-5.1,-0.5) (-5.4,0.7)--(-5.4,0) (-5.4,0.7)--(-5.7,0) (-5.4,0.7)--(-5.1,0)  (-5.4,-1.4)--(-5.4,-0.5)  (-5.4,-1.4)--(-5.7,-0.5)  (-5.4,-1.4)--(-5.1,-0.5);
  \node [XY] at (-5.4,0) {};
  \node [XY] at (-5.4,-0.7) {};
  \node [simple,label=180:$p'_1$] at (-5.4,0.7) {};
  \node [simple,label=180:$q'_1$] at (-5.4,-1.4) {};
  \draw (-4.2,0)--(-4.2,-0.5)  (-4.5,0)--(-4.5,-0.5) (-3.9,0)--(-3.9,-0.5) (-4.2,0.7)--(-4.2,0) (-4.2,0.7)--(-4.5,0) (-4.2,0.7)--(-3.9,0)  (-4.2,-1.4)--(-4.2,-0.5)  (-4.2,-1.4)--(-4.5,-0.5)  (-4.2,-1.4)--(-3.9,-0.5);
  \node [XY] at (-4.2,0) {};
  \node [XY] at (-4.2,-0.7) {};
  \node [simple,label=180:$p'_2$] at (-4.2,0.7) {};
  \node [simple,label=180:$q'_2$] at (-4.2,-1.4) {};
  \draw (-3,0)--(-3,-0.5)  (-3.3,0)--(-3.3,-0.5) (-2.7,0)--(-2.7,-0.5) (-3,0.7)--(-3,0) (-3,0.7)--(-3.3,0) (-3,0.7)--(-2.7,0)  (-3,-1.4)--(-3,-0.5)  (-3,-1.4)--(-3.3,-0.5)  (-3,-1.4)--(-2.7,-0.5);
  \node [XY] at (-3,0) {};
  \node [XY] at (-3,-0.7) {};
  \node [simple,label=180:$p'_3$] at (-3,0.7) {};
  \node [simple,label=180:$q'_3$] at (-3,-1.4) {};
  
  \draw (-1.5,0)--(-1.5,-0.5)  (-1.8,0)--(-1.8,-0.5) (-1.2,0)--(-1.2,-0.5) (-1.5,0.7)--(-1.5,0) (-1.5,0.7)--(-1.8,0) (-1.5,0.7)--(-1.2,0)  (-1.5,-1.4)--(-1.5,-0.5)  (-1.5,-1.4)--(-1.8,-0.5) (-1.5,-1.4)--(-1.2,-0.5);
  \node [XY] at (-1.5,0) {};
  \node [XY] at (-1.5,-0.7) {};
  \node [simple,label=180:$p'_4$] at (-1.5,0.7) {};
  \node [simple,label=180:$q'_4$] at (-1.5,-1.4) {};
  \node at (-0.75,-0.35) {$\ldots$};
  \draw (-0.3,0)--(-0.3,-0.5)  (0,0)--(0,-0.5) (0.3,0)--(0.3,-0.5) (0,0.7)--(-0.3,0)  (0,0.7)--(0,0)  (0,0.7)--(0.3,0) (0,-1.4)--(-0.3,-0.5) (0,-1.4)--(0,-0.5)  (0,-1.4)--(0.3,-0.5);
  \node [XY] at (0,0) {};
  \node [XY] at (0,-0.7) {};
  \node [simple,label=180:$p'_{k-1}$] at (0,0.7) {};
  \node [simple,label=180:$q'_{k-1}$] at (0,-1.4) {};
  
  \node [simple,label=180:$w$] at (-2.85,1.7) {};
  \node [simple,label=180:$t$] at (-2.85,2.3) {};
  \draw (-2.85,1.7)--(-5.4,0.7) (-2.85,1.7)--(-4.2,0.7) (-2.85,1.7)--(-3,0.7)  (-2.85,1.7)--(-1.5,0.7)  (-2.85,1.7)--(0,0.7) (-2.85,1.7)--(-2.85,2.3);
  
  \node [simple,label=180:$u$] at (-4.2,-2.4) {};
  \node [simple,label=180:$v$] at (-0.75,-2.4) {};
  \draw (-4.2,-2.4)--(-5.4,-1.4)  (-4.2,-2.4)--(-4.2,-1.4) (-4.2,-2.4)--(-3,-1.4)  (-0.75,-2.4)--(-1.5,-1.4)  (-0.75,-2.4)--(0,-1.4)  (-4.2,-2.4)--(-4.2,-3.4)  (-0.75,-2.4)--(-0.75,-3.4);

  \draw (-2, -3.4)--(-2.95, -3.4)  (-2, -3.1)--(-2.95, -3.1)  (-2, -3.7)--(-2.95, -3.7)  (-4.2,-3.4)--(-2.95,-3.4)  (-4.2,-3.4)--(-2.95,-3.1) (-4.2,-3.4)--(-2.95,-3.7)  (-0.75,-3.4)--(-2,-3.4)  (-0.75,-3.4)--(-2,-3.1) (-0.75,-3.4)--(-2,-3.7);
  \node [simple,label=180:$q'$] at (-4.2,-3.4) {};
  \node [simple,label=0:$p'$] at (-0.75,-3.4) {};
  \node [YX] at (-2, -3.4) {};
  \node [YX] at (-2.95, -3.4) {};

  \end{tikzpicture}
\caption{Gadget $G'_k$ in the reduction for line graphs.}
\label{fig:linegadget}
\end{center}
\end{figure}
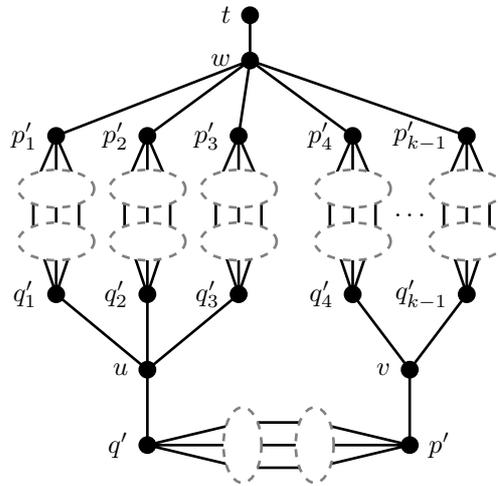


\begin{lemma}
Let $k\ge 7$, and $G'_k(p,q)$ be an edge-gadget obtained as above. Then the following hold:
\begin{enumerate}
  \item In every $k$-edge coloring of $G'_k(p,q)$, edges $pp'$ and $qq'$ receive the same color; and
  \item For every $\alpha\in\{1,\ldots,k\}$, there exists a $(pp',\alpha)$-ECGC of $G'_k(p,q)$.
\end{enumerate}
\end{lemma}
Furthermore, let $G'_k$ be constructed as before. Then, $G'_k$ is triangle-free $k$-edge-chromatic graph and the following holds.
\begin{enumerate}
  \item[3.] $G'_k$ has a $(wt,\alpha)$-CGC with $k$ colors if and only if $\alpha\le k-3$.
\end{enumerate}
\begin{proof}
Because $G'_k(p,q)$ is a bipartite graph, it has no triangles and, by K\"onig's line coloring Theorem, we know that $\chi'(G'_k(p,q)) = \Delta(G'_k(p,q)) = k$. In fact, observe that the only perfect matching containing $pp'$ must also contain $qq'$, which gives us Property (1) below. Property (2) is also easy to be verified.

Now, by Property (1), the colors in $\{wp'_i\mid i\in\{1,\ldots,k-1\}\}$ are the same as the colors in $\{uq'_1,uq'_2,uq'_3\}\cup\{vq'_i\mid i\in\{4,\ldots,k-1\}\}$. This means that, in every $k$-edge coloring of $G'_k$, edge $wt$ must get the same color as edges $uq'$ and $vp'$. Because edge $uq'$ has only 3 adjacent edges that are incident in $u$, and edge $vp'$ has only $k-4$ edges that are incident in $v$, we get that if we start to color $G'_k$ in $wt$ with color $\alpha$, we end up using $k$ colors only if $\alpha\le \max\{4,k-3\} = k-3$ (recall that $k\ge 7$). Using Property (2), one can verify that a $(wt,\alpha)$-ECGC exists when $\alpha \le k-3$. 
\end{proof}

We are finally ready to finish our proof. So, let $G$ be a triangle-free cubic graph, and let $G^*$ be obtained from $G$ by appending $k-3$ copies of $G'_k$ on each vertex of $G$. If $\chi'(G)=3$, then a ECGC of $G^*$ with $k$ colors can be constructed as follows. Let $(v_1,\ldots,v_n)$ be a connected order of $V(G)$ and let $f$ be a 3-edge-coloring of $G$ that uses colors $\{k-2,k-1,k\}$. For each $i\in\{1,\ldots,n\}$, color the copies of $G'_k$ incident to $v_i$ using colors 1 through $k-3$ on the appended edges (this is possible by Property (3)), then color the uncolored edges incident to $v_i$ in ascending order of their colors. 
Now, suppose that $\chi'_c(G^*) = k$. By Property (3) and the fact that $d_{G^*}(v) = k$ for every $v\in V(G)$, we get that the colors 1 through $k-3$ are all used in the copies of $G'_k$ appended on $v$. This means that the edges of $G$ can only use colors $\{k-2,k-1,k\}$, which gives us a 3-edge-coloring of $G$.

\section{$P_9$-free graphs}\label{sec:paths}

In this section, we prove Lemma~\ref{lem:paths}. In~\cite{huang2016improved}, it is proved that deciding $\chi(G)\le 4$ is $\NP$-complete when $G$ is a $P_9$-free graph. Observe that if $G$ is a $P_9$-free graph and $G'$ is obtained from $G$ by adding a universal vertex, then $G'$ is also $P_9$-free and is such that $\chi(G)\le 4$ if and only if $\chi(G')\le 5$. Therefore, deciding whether $\chi(G)\le 5$ is $\NP$-complete even if $G$ is a $P_9$-free graph with a universal vertex. We reduce this problem to the problem of deciding whether $\chi_c(G)=\chi(G)$ for $P_9$-free graphs. For shortness, we call the former problem \texttt{5-Col $P_9$-free} and the latter \texttt{CGC $P_9$-free}. 

The following lemma proved in~\cite{BCD+14} will be useful.

\begin{lemma}\label{lem:toolLemma}
Let $G$ be a connected graph, $v\in V(G)$ and $\alpha$ be any positive integer. Then, there exists a $(v,\alpha)$-CGC of $G$ with at most $\max\{\alpha,\chi(G)+1\}$ colors.
\end{lemma}

In our  reduction, we use an auxiliary graph $H$ with the following properties.
\begin{enumerate}
\item $H$ is connected, $\chi(H)=5$ and for every $\alpha\in\{1,\cdots,5\}$ and every $u\in V(H)$, there is no $(u,\alpha)$-CGC of $H$ with~5 colors;
\item There exists $v\in V(H)$ such that there are no induced $P_8$ with extremity in $v$;
\item $H$ is a $P_9$-free graph.
\end{enumerate}

We first assume that such a graph exists, and later we show how to construct it. Let $G$ be an instance of \texttt{5-Col $P_9$-free} and let $u\in V(G)$ be its universal vertex. Also, let $G'$ be obtained from $G+H$ by identifying vertices $u$ and $v$; denote the resulting vertex by $w$, the vertices corresponding to $G$ by $V_G$ and the vertices corresponding to $H$ by $V_H$. Because of Properties (2) and (3), the fact that $G$ is $P_9$-free and that $w$ dominates $V_G$ in $G'$, we know that $G'$ is also a $P_9$-free graph. We claim that $\chi_c(G')=\chi(G')$ if and only if $\chi(G)>5$. This finishes the proof.

First, let $\chi(G)=k>5$, and observe that, because $\chi(H)=5$, we get that $\chi(G')=k>5$. Let $f$ be an optimal coloring of $G$. Since $u$ is universal in $G$, we know that $\chi'_c(G)=\chi(G)$; it suffices to start by coloring $u$ with~1 and using any optimal greedy ordering for $G-u$ (recall that it always exists). Also, because $\chi(H)+1\le k$, by Lemma~\ref{lem:toolLemma} we know that there exists a $(v,f(v))$-CGC of $H$ with at most $k$ colors. This means that $f$ can be extended to a connected greedy coloring of $G'$ with $\chi(G')$ colors.

Now, suppose that $\chi(G)\le 5$ and let $f$ be an optimal connected greedy coloring of $G'$. Note that, since $\chi(H)=5$, we get that $\chi(G') = 5$. Also, note that either $f$ starts in $V_H$, in which case $f$ restricted to $H$ is a CGC of $H$, or $f$ starts in $V_G\setminus\{w\}$, in which case $f$ restricted to $H$ is a $(v,f(v))$-CGC of $H$. In both situations, we get that the number of used colors is bigger than~5 by Property (1). 

It remains to construct the desired auxiliary graph. 
Observe Figure~\ref{fig:P9free} to follow the construction. We start with a complete graph on vertices $K=\{v_1,\ldots,v_5\}$. Then, for each $i\in \{1,\ldots,5\}$, add vertices $x_i,y_i$ and make each adjacent to two vertices of $K\setminus\{v_i\}$ in a way that they do not have common neighbors. Finally, again for each $i\in \{1,\ldots,5\}$, add a clique $K^i$ of size~4 and make both $x_i$ and $y_i$ complete to $K^i$. In the following lemma we prove that $H$ satisfies the desired conditions.

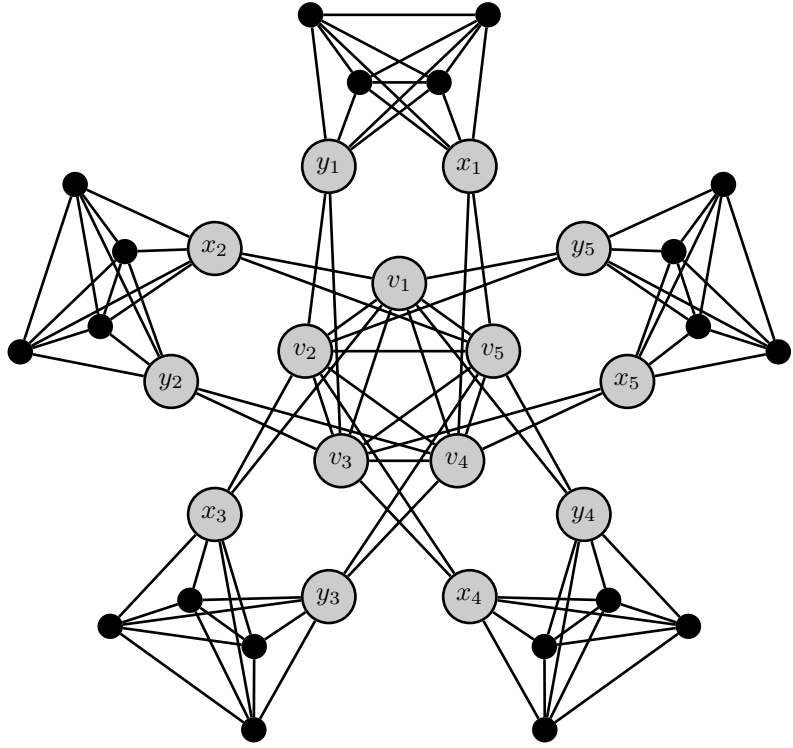
\begin{figure}[thb]
\begin{center}
  \begin{tikzpicture}
  \pgfsetlinewidth{1pt}
  \tikzset{vertex/.style={circle, minimum size=0.7cm, fill=black!20, draw, inner sep=0.5pt}}
  \tikzset{peq/.style={circle, minimum size=0.3cm, fill=black, draw, inner sep=1pt}}
  
 \foreach \i in {1,...,5}{
    \node [vertex] (x\i) at (\i*72:3cm) {$x_\i$};
    \node [vertex] (y\i) at (\i*72+36:3cm) {$y_\i$};
    \node [vertex] (v\i) at (\i*72+18:1.3cm) {$v_\i$};}

\draw (v1)--(v2)--(v3)--(v4)--(v5)--(v2)--(v4)--(v1)--(v3)--(v5)--(v1)
(x1)--(v4)  (x1)--(v5)  (y1)--(v2)  (y1)--(v3)
(x2)--(v1)  (x2)--(v5)  (y2)--(v3)  (y2)--(v4)
(x3)--(v1)  (x3)--(v2)  (y3)--(v4)  (y3)--(v5)
(x4)--(v2)  (x4)--(v3)  (y4)--(v1)  (y4)--(v5)
(x5)--(v3)  (x5)--(v4)  (y5)--(v1)  (y5)--(v2);

  \foreach \j in {1,...,5}{
        \node [peq] (o\j1) at (72*\j+4.5:5cm) {};
        \node [peq] (o\j2) at (72*\j+10.5:4cm) {};
        \node [peq] (o\j3) at (72*\j+25.5:4cm) {};
        \node [peq] (o\j4) at (72*\j+31.5:5cm) {};
        \foreach \i in {1,...,4} { \draw (x\j)--(o\j\i)--(y\j);}
        \draw (o\j1)--(o\j2)--(o\j3)--(o\j4)--(o\j1)   (o\j2)--(o\j4)  (o\j1)--(o\j3);
  }

  \end{tikzpicture}
\caption{$P_9$-free gadget.}
\label{fig:P9free}
\end{center}
\end{figure}

\begin{lemma}
Let $H$ be the graph obtained as explained above. Then, Properties (1)-(3) hold for $H$.
\end{lemma}
\begin{proof}
First, we prove Property (1). Clearly, $H$ is a connected graph. To see that $H$ is 5-chromatic, observe that $\chi(H)\ge 5$ and that we can color $H$ with~5 colors by giving color $i$ to $\{v_i,x_i,y_i\}$ and colors $\{1,\ldots,5\}\setminus\{i\}$ to $K^i$, for each $i\in \{1,\ldots,5\}$. It remains to show that $H$ has no $(v,\alpha)$-CGC, for every $v\in V(G)$ and every $\alpha\in\{1,\ldots,5\}$. By contradiction, consider $f$ to be such a coloring and let $\pi$ be the corresponding connected order. For each $i\in \{1,\ldots,5\}$, denote by $H_i$ the subgraph $H[K^i\cup\{x_i,y_i\}]$. Note that in every 5-coloring of $H$, vertices $x_i$, $y_i$ and $v_i$ must have the same color, for every $i\in \{1,\ldots,5\}$. This means that there exist distinct $i,j\in \{1,\ldots,5\}$ such that $f(x_i)=4$ and $f(x_j)=5$. Also, we know that $v$ is either not in $V(H_i)$ or not in $V(H_j)$, say $v\notin V(H_i)$. Since $\pi$ is a connected order, the first vertex of $V(H_i)$ to appear in $\pi$ is either $x_i$ or $y_i$, say it is $x_i$. We get a contradiction since at this point $x_i$ has at most two colored neighbors, which means that the color of $x_i$ would be at most~3. 

Now, observe that any induced path $P$ in $H$ has at most two vertices of $K$. If $P$ starts with a vertex of $K$, say $v_1$, then it has at most~5 vertices, namely some other vertex $v_j\in K$ and vertices $\{x_1,w,y_1\}$ for some $w\in K^1$. This ensures Property (2). Now, suppose that $P$ has size bigger than~5, which means that both of its extremities are not in $K$. Because the longest induced path in $V(H_i)$ has length~3 for every $i\in \{1,\ldots,5\}$, and because $P$ contains at most~2 vertices of $K$, we get that $P$ has at most~8 vertices and Property (3) holds.
\end{proof}

Observe that, because \texttt{$k$-Col $P_\ell$-free} is $\NP$-complete for $k\ge 5$ and $\ell=6$, and for $k\ge 4$ and $\ell=7$~\cite{huang2016improved}, if we manage to obtain a similar gadget $H$ for these values, then the same arguments work. More formally, given positive integers $k$ and $\ell$, we say that a graph $H$ is a \emph{$(k,\ell)$-gadget} if:
\begin{enumerate}
\item $H$ is connected, $\chi(H)=k$ and for every $\alpha\in\{1,\cdots,k\}$ and every $u\in V(H)$, there is no $(u,\alpha)$-CGC of $H$ with~$k$ colors;
\item There exists $v\in V(H)$ such that there are no induced $P_{\ell-1}$ with extremity in $v$;
\item $H$ is a $P_\ell$-free graph.
\end{enumerate}

The results in~\cite{huang2016improved} and the above proof gives us the following meta-theorem (recall that we need to add a universal vertex in the beginning of the proof, this is why $k$ increases below when related to the results in~\cite{huang2016improved}).

\begin{theorem}
If there exists a $(k,7)$-gadget for any $k\ge 5$, then deciding $\chi_c(G)=\chi(G)$ is $\NP$-hard for $P_7$-free graphs. Also, if there exists a $(k,6)$-gadget for any $k\ge 6$, then deciding $\chi_c(G)=\chi(G)$ is $\NP$-hard for $P_6$-free graphs.
\end{theorem}

As already mentioned in the introduction, in~\cite{babel1994hard} a graph satisfying Property (1) above is called globally hard-to-color. The cited article is entirely dedicated to finding small hard-to-color graphs, which shows that finding the desired gadgets is no trivial task.



\section{Equality on $P_5$-free and $(P_4+K_1)$-free graphs}\label{sec:P5}

It is known that if $G$ is $P_4$-free, then $G$ is $(\chi,\Gamma)$-perfect~\cite{ChSe79}, which implies that $\chi(G) = \chi_c(G) = \Gamma_c(G)=\Gamma(G)$; so in our proof we can suppose that $G$ has some $P_4$. To prove Lemma~\ref{lem:poly}, we pick some dominating $P_4$ in $G$ and an optimal coloring of $G$ and change the coloring in order to prove that a CGC can be obtained. In what follows, given a coloring $f$ of $G$ and a vertex $v\in V(G)$, the \emph{$(1,2)$-component containing $v$} is the component of the subgraph induced by color classes 1 and 2 that contains $v$. 

Now, before we proceed, observe that if $f$ is an greedy coloring of $G$ with $\chi(G)$ colors, and $i,j$ are arbitrary colors, with $i<j$, then a greedy coloring $f'$ of $G$ with at most $\chi(G)$ colors exist where the color classes $i$ and $j$ are switched in $f'$. It suffices to recolor every $v\in f^{-1}(i)$ with $j$, every $v\in f^{-1}(j)$ with $i$, then, for every $\ell\in \{i+1,\ldots,\chi(G)\}$ and every $v\in f^{-1}(\ell)$, move $v$ to color class $h$, where $h$ is the minimum value such that $N(v)\cap f^{-1}(h)=\emptyset$.

The following easy proposition is crucial in the proof.

\begin{proposition}\label{prop:dominating}
Let $G$ be a graph, $X\subseteq G$ be a connected dominating set of $G$ and $f$ be a $k$-coloring of $G$. If $f$ restricted to $X$ is a CGC of $G[X]$, then $\chi_c(G)\le k$.
\end{proposition}
\begin{proof}
It suffices to prove that we can obtain a connected ordering of $V(G)$ that produces a CGC with at most $k$ colors. For this, suppose without loss of generality, that $f$  is a greedy coloring of $G$ (otherwise, simply decrease the color of vertices that are not greedily colored). Start with a connected order of $X$ that produces $f$, then add $f^{-1}(i)\cap (V(G)\setminus X)$, for each $i\in \{1,\ldots, k\}$, in this order. The obtained order is a connected one that produces $f$ because every $u\in V(G)\setminus X$ is adjacent to some vertex in $X$, each vertex in $X$ receives the same color as in $f$, and the subsequent vertices are ordered according to their color. 
\end{proof}

Now, consider $G$ to be a connected graph which is either $P_5$-free or $(P_4+K_1)$-free. First, we want to prove that $G$ has a dominating $P_4$. Observe that any $P_4$ in a connected $(P_4+K_1)$-free graph is a dominating $P_4$. So, assume that $G$ is a $P_5$-free graph. 
By a result in~\cite{BaTu90}, we know that $G$ has a dominating set $C$ that is either a clique or a path on~3 vertices. 
Let $f$ be a greedy coloring of $G$ that uses $\chi(G)$ colors (recall that this is always possible). Proposition~\ref{prop:dominating} directly implies that $\chi_c(G)=\chi(G)$ if $C$ is a clique, since we can suppose without loss of generality that the colors of $C$ in $f$ are $\{1,\ldots,|C|\}$. Therefore, we can suppose that $C$ is a dominating $P_3$. 
In fact, we can suppose that there exists $w\in N(z)\setminus N(\{x,y\})$, as otherwise $\{x,y\}$ is a dominating clique and the lemma follows from Proposition~\ref{prop:dominating}.

Now, consider a dominating $P_4$, $P=(x,y,z,w)$, in $G$ and let $f$ be a greedy coloring of $G$ that uses $\chi(G)$ colors. 
We obtain a subset $X$ containing $V(P)$ that satisfies the conditions of Proposition~\ref{prop:dominating}. Clearly, if $f$ restricted to $P$ uses only~2 colors, then we can suppose, without loss of generality, that $f(x)=(z)=2$ and $f(y) = f(w)=1$, in which case the subset $V(P)$ itself satisfies the desired conditions. So suppose otherwise and consider the following cases:

\begin{enumerate}
  \item $f(V(P)) = \{1,2,3\}$: one can see that, without loss of generality, only two possible colorings exist: 
  \begin{enumerate} 
    \item $x$ and $w$ have the same color: suppose, without loss of generality, that $(x,y,z,w)$ are colored with $(1,2,3,1)$, in this order. Let $C$ be the $(1,2)$-component containing $w$. If $x\in C$, then $C\cup\{z\}$ is the desired subset. Otherwise, we switch colors 1 and 2 in $C$ to have a situation similar to the next subcase. \\
    \item $x$ and $z$ have the same color: suppose, without loss of generality, that $(x,y,z,w)$ are colored with $(1,2,1,3)$, in this order. Let $N$ be the set $N(w)\cap f^{-1}(2)$ (neighbors of $w$ colored with~2). Because $f$ is a greedy coloring, we know that $N\neq \emptyset$. Now, if there exists $u\in N$ such that the $(1,2)$-component $C$ containing $u$ also contains $x$, then again $C\cup\{w\}$ is the desired set. Otherwise, note that $C$ cannot contain $z$ either, and that we can switch colors 1 and 2 in every $(1,2)$-component containing some vertex of $N$ in order to recolor $w$ with~2, thus making $V(P)$ the desired subset;
  \end{enumerate}
  \item $f(V(P)) = \{1,2,3,4\}$: without loss of generality, suppose that $P$ is colored with $(4,3,1,2)$, in this order. We can suppose that there exists $u\in N(y)$ such that $f(u)=2$ and the $(1,2)$-component $C$ containing $u$ also contains some neighbor of $y$ colored with~1; this is because otherwise we can recolor vertices in order to decrease the number of colors appearing in $P$, thus getting one of the previous cases. Similarly, there exists $v\in N(x)$ such that $f(v)=2$ and the $(1,2)$-component $C'$ containing $v$ also contains some neighbor of $x$ of color~1, say $v'$. Now, if $v'$ is adjacent to $y$, $z$ or $w$, we get that $C\cup C'\cup V(P)$ is the desired subset: it suffices to greedily color $C$, $y$, $z$ (if not yet colored), $w$, $v'$, $C'$, and $x$, in this order. Otherwise, because $(v',x,y,z,w)$ is a $P_5$, we get that $G$ is a $(P_4+K_1)$-free graph, in which case every $P_4$ of $G$ is a dominating $P_4$. Thus, the argument for the 3-colored dominating $P_4$, $(v',x,y,z)$ can be applied.

\end{enumerate}


\bibliographystyle{plain}
\bibliography{biblio}

\end{document}